\documentclass[pdflatex,sn-mathphys-num]{sn-jnl}

\usepackage{graphicx}%
\usepackage{multirow}%
\usepackage{amsmath,amssymb,amsfonts}%
\usepackage{amsthm}%
\usepackage{mathrsfs}%
\usepackage[title]{appendix}%
\usepackage{xcolor}%
\usepackage{textcomp}%
\usepackage{manyfoot}%
\usepackage{booktabs}%
\usepackage{algorithm}%
\usepackage{algorithmicx}%
\usepackage{algpseudocode}%
\usepackage{listings}%
\usepackage{array}

\usepackage{tikz}
\usetikzlibrary{
  arrows.meta,
  calc,
  positioning,
  decorations.pathreplacing,
  patterns,
  fit,
  backgrounds,
  intersections,
  shapes.geometric  
}

\hypersetup{
  colorlinks=true,
  linkcolor=blue!60!black,
  citecolor=blue!60!black,
  urlcolor=blue!60!black
}

\newtheorem{theorem}{Theorem}[section]

\newtheorem{corollary}[theorem]{Corollary}
\theoremstyle{definition}
\newtheorem{definition}[theorem]{Definition}
\newtheorem{remark}[theorem]{Remark}

\newtheorem{solution*}{Solution}[section]
\newcommand{\R}{\mathbb{R}}
\newcommand{\argmin}{\operatorname*{arg\,min}}
\newcommand{\dist}{\operatorname{dist}}

\newcommand{\cF}{\mathcal{F}}
\newcommand{\cS}{\mathcal{S}}

\newcommand{\norm}[1]{\left\lVert #1 \right\rVert}
\newcommand{\abs}[1]{\left\lvert #1 \right\rvert}
\newcommand{\set}[1]{\left\{ #1 \right\}}

\newcommand{\ip}[2]{\left\langle #1,#2\right\rangle}
\newcommand{\pos}[1]{\left(#1\right)_{+}}

\title{Introduction to Exact Penalization for Mathematical Programming with Equilibrium Constraints}
\author{Louis Shuo Wang\thanks{Email: wang.s41@northeastern.edu}}
\affil{Department of Mathematics, Northeastern University, Boston, 02115, MA, USA}

\abstract{
We present a focused introduction to exact penalty methods for nonlinear programs and mathematical programs with equilibrium constraints (MPECs), emphasizing their connection to modern error bound theory. The goal is twofold. First, we explain how classical optimality conditions can be interpreted through exact penalization, and why such results typically rely on constraint regularity conditions that can be understood as error bounds on perturbations of feasible sets. We then highlight how recent developments based on subanalytic geometry and Łojasiewicz-type inequalities extend this framework beyond classical regularity assumptions, enabling exact penalization under broader analytic conditions.

Second, we demonstrate how this theory can be applied in practice to MPECs by reformulating them via KKT systems and constructing exact penalty functions based on residual mappings. Particular attention is given to fractional-order penalties arising from Łojasiewicz error bounds, as well as to improved formulations for special problem classes where sharper exponents can be obtained. These developments provide both theoretical insight and practical guidance for analyzing and solving challenging constrained optimization problems.}

\begin{document}

\maketitle

\section{Introduction and motivation}

A mathematical program with equilibrium constraints (MPEC) is an optimization problem in which feasibility is partly defined by the solution set of a lower-level equilibrium system, often a variational inequality (VI), complementarity system, or generalized equation. MPECs are difficult for at least two structural reasons:

\begin{enumerate}[label=\arabic*.]
    \item the feasible set is defined \emph{implicitly} through an equilibrium problem;
    \item natural one-level reformulations introduce \emph{complementarity constraints}, which destroy standard constraint qualifications and invalidate much of classical nonlinear programming theory.
\end{enumerate}

A central idea in exact penalization is to replace hard equilibrium constraints by a residual term added to the objective. If the penalty is \emph{exact}, then minimizers of the penalized problem coincide with minimizers of the original constrained problem for all sufficiently large penalty parameters. In classical nonlinear programming this is often achieved with an $\ell_1$-type penalty under a regularity condition. For MPECs, however, the correct penalty may involve a \emph{fractional power} of a residual. This reflects a weaker error bound and is one of the main conceptual messages of the subject.

These notes develop the following chain of ideas:

\[
\text{exact penalty}
\Longleftrightarrow
\text{error bound}
\Longleftrightarrow
\text{metric control of feasibility violation}.
\]

The modern viewpoint relevant here is:
\begin{enumerate}[label=(\roman*)]
    \item represent the MPEC as a standard constrained problem via KKT conditions for the lower-level equilibrium system;
    \item define a residual for those KKT conditions;
    \item invoke a Łojasiewicz-type error bound on a subanalytic set;
    \item deduce exact penalization, often with exponent $\gamma \in (0,1]$.
\end{enumerate}

\section{The MPEC model}
We work with a prototypical MPEC of the form
\begin{equation}
\label{eq:mpec}
\begin{aligned}
    \min_{x,y}\quad & f(x,y) \\
    \text{s.t.}\quad & (x,y)\in Z,\\
    & y \in \cS(x),
\end{aligned}
\end{equation}
where
\begin{enumerate}[label=(\roman*)]
    \item $x\in\R^n$ is the upper-level decision variable,
    \item $y\in\R^m$ is the lower-level or equilibrium variable,
    \item $Z\subseteq \R^n\times\R^m$ is a given constraint set,
    \item $\cS(x)$ is the solution map of a parametric equilibrium problem.
\end{enumerate}

A standard example is a parameterized VI:
\begin{equation}
\label{eq:vi}
\text{find } y\in C(x)\ \text{ such that }\ \ip{F(x,y)}{v-y}\ge 0\quad \forall v\in C(x),
\end{equation}
where
\[
C(x) := \set{y\in\R^m : g_i(x,y)\le 0,\ i=1,\dots,\ell }.
\]

If the lower-level problem admits a KKT representation, then \eqref{eq:vi} can be expressed using a multiplier vector $\lambda\in\R^\ell_+$:
\begin{equation}
\label{eq:kkt-vi}
F(x,y)+\sum_{i=1}^\ell \lambda_i \nabla_y g_i(x,y)=0,\qquad
g_i(x,y)\le 0,\qquad
\lambda_i g_i(x,y)=0,\qquad
\lambda_i\ge 0.
\end{equation}

This converts \eqref{eq:mpec} into a one-level constrained problem in variables $(x,y,\lambda)$.

\section{Residual functions and the exact-penalty principle}

\subsection{Residuals for ordinary constrained optimization}

Consider the nonlinear program
\begin{equation}
\label{eq:nlp}
\begin{aligned}
    \min_x\quad & \theta(x)\\
    \text{s.t.}\quad & x\in X,\\
    & g(x)\le 0,\qquad h(x)=0,
\end{aligned}
\end{equation}
where $X\subseteq \R^n$, $g:\R^n\to\R^p$, $h:\R^n\to\R^q$, and $\theta:\R^n\to\R$.

A natural residual for the functional constraints is
\begin{equation}
\label{eq:nlp-residual}
r(x):=\sum_{i=1}^p \pos{g_i(x)}+\sum_{j=1}^q \abs{h_j(x)}.
\end{equation}
Then $r(x)\ge 0$ and
\[
r(x)=0 \iff x \text{ is feasible for } \eqref{eq:nlp}.
\]

\begin{definition}[Exact penalty equivalent]
A penalized problem
\[
\min_{x\in X}\ \theta(x)+\alpha \psi(x)
\]
is called an \emph{exact penalty equivalent} of \eqref{eq:nlp} if, for all sufficiently large $\alpha$, its set of global minimizers coincides with the set of global minimizers of \eqref{eq:nlp}.
\end{definition}

The crucial issue is not merely that $\psi(x)=0$ on the feasible set, but that $\psi$ must dominate the distance to feasibility in a suitable quantitative way.

\subsection{Metric idea}

Let
\[
W:=\set{x\in X: g(x)\le 0,\ h(x)=0}.
\]
Suppose $\theta$ is Lipschitz on $X$ with modulus $L$, and suppose there exists a nonnegative function $\psi$ such that
\[
\dist(x,W)\le c\,\psi(x)\qquad \forall x\in X.
\]
Then one expects
\[
\theta(x)+\alpha \psi(x)
\]
to become exact once $\alpha > cL$ (up to a factor relating global bounds on $\psi$). Thus:

\begin{center}
\emph{Exact penalization is driven by error bounds.}
\end{center}

\section{Subanalyticity and Łojasiewicz-type inequalities}

The theory in this topic is substantially broader than smooth nonlinear programming. The right ambient class is often the class of \emph{subanalytic} sets and functions.

\begin{definition}[Analytic function]
A function $\varphi:U\to\R$ on an open set $U\subseteq \R^n$ is analytic if it admits a convergent power-series expansion in a neighborhood of each point of $U$.
\end{definition}

\begin{definition}[Semianalytic and subanalytic sets]
A set $X\subseteq \R^n$ is \emph{semianalytic} if locally it can be described by finitely many analytic equalities and inequalities. A set is \emph{subanalytic} if locally it is the projection of a bounded semianalytic set in a higher-dimensional space.
\end{definition}

Subanalyticity is large enough to include many piecewise-analytic constructions and many nonsmooth residuals arising from maxima, absolute values, and finite-dimensional equilibrium reformulations.

The key analytical tool is a Łojasiewicz-type inequality.

\begin{theorem}[Łojasiewicz-type comparison inequality]
\label{thm:loja}
Let $S\subseteq \R^n$ be compact and subanalytic, and let $\phi,\psi:S\to\R$ be continuous and subanalytic. If
\[
\phi^{-1}(0)\subseteq \psi^{-1}(0),
\]
then there exist $\rho>0$ and an integer $N^\ast\ge 1$ such that
\[
\rho\,\abs{\psi(x)}^{N^\ast}\le \abs{\phi(x)}\qquad \forall x\in S.
\]
\end{theorem}

A particularly important specialization is obtained by taking
\[
\phi(x)=r(x),\qquad \psi(x)=\dist(x,W),
\]
which yields an error bound of the form
\begin{equation}
\label{eq:generic-error-bound}
\rho\,\dist(x,W)\le r(x)^{1/N^\ast}.
\end{equation}

Thus, a fractional-power penalty appears naturally from the geometry of subanalytic feasibility sets.

\section{General exact penalty theorem}

We now state the basic template.

\begin{theorem}[General exact penalty theorem]
\label{thm:general-exact}
Consider \eqref{eq:nlp}. Assume:
\begin{enumerate}[label=(\roman*)]
    \item $X\subseteq \R^n$ is compact and subanalytic;
    \item $\theta$ is Lipschitz continuous on $X$;
    \item each $g_i$ and $h_j$ is continuous and subanalytic;
    \item the feasible set $W$ is nonempty.
\end{enumerate}
Then there exist $\alpha^\ast>0$ and an integer $N^\ast\ge 1$ such that for every $\alpha\ge \alpha^\ast$ and every $N\ge N^\ast$,
\[
\argmin_{x\in W}\theta(x)
=
\argmin_{x\in X}\Bigl(\theta(x)+\alpha r(x)^{1/N}\Bigr),
\]
where $r$ is the residual \eqref{eq:nlp-residual}.
\end{theorem}

\begin{proof}[Proof sketch]
Let
\[
\psi(x):=\dist(x,W).
\]
By subanalyticity and compactness, Theorem~\ref{thm:loja} implies that
\[
\rho\,\psi(x)\le r(x)^{1/N^\ast}\qquad \forall x\in X
\]
for some $\rho>0$. If $z\in W$ is a nearest feasible point to $x$, then Lipschitz continuity gives
\[
\theta(x)\ge \theta(z)-L\norm{x-z}=\theta(z)-L\psi(x).
\]
Hence
\[
\theta(x)+\alpha r(x)^{1/N^\ast}
\ge
\theta(z)+(\alpha\rho-L)\psi(x).
\]
Thus whenever $\alpha\rho>L$, the penalized objective cannot improve on a feasible comparator unless $\psi(x)=0$, i.e., unless $x\in W$. This proves exactness.
\end{proof}

\begin{remark}
The important point is not the precise exponent $1/N^\ast$, but the mechanism:
\[
\text{error bound} \Longrightarrow \text{exact penalty}.
\]
When classical regularity is strong enough, one may obtain $N^\ast=1$. In degenerate settings, only a weaker fractional exponent may be available.
\end{remark}

\section{Local exact penalties}

The same logic localizes. If $x^\ast\in W$ is a local minimizer of \eqref{eq:nlp}, then on a sufficiently small neighborhood $B(x^\ast,\varepsilon)$ the restricted problem is a global minimization problem. Applying the previous theorem locally yields:

\begin{corollary}[Local exact penalty characterization]
\label{cor:local}
Assume $X$ is closed and subanalytic, each $g_i,h_j$ is continuous and subanalytic, and $\theta$ is Lipschitz near a feasible point $x^\ast\in W$. Then $x^\ast$ is a local minimizer of \eqref{eq:nlp} if and only if, locally on $X$, it minimizes
\[
\theta(x)+r(x)^{1/N}
\]
for all sufficiently large integers $N$.
\end{corollary}

A useful message for research students is that local and global exact penalization are structurally similar, but the global version needs compactness or some coercivity substitute, whereas the local version only needs control on a neighborhood.

\section{From equilibrium constraints to KKT residuals}

We now specialize to the MPEC \eqref{eq:mpec}. Assume the lower-level feasible set is
\[
C(x)=\set{y\in\R^m : g_i(x,y)\le 0,\ i=1,\dots,\ell},
\]
and the lower-level equilibrium is expressed by a VI with mapping $F(x,y)$.

Under suitable assumptions, the VI admits the KKT representation
\[
F(x,y)+\sum_{i=1}^\ell \lambda_i \nabla_y g_i(x,y)=0,\qquad
g_i(x,y)\le 0,\qquad
\lambda_i\ge 0,\qquad
\lambda_i g_i(x,y)=0.
\]
Then the MPEC can be written as
\[
\min_{x,y,\lambda}\ f(x,y)
\]
subject to $(x,y,\lambda)$ belonging to a tractable set and satisfying the KKT equalities/inequalities.

This motivates the KKT residual
\begin{equation}
\label{eq:mpec-residual}
r(x,y,\lambda):=
\norm{F(x,y)+\sum_{i=1}^{\ell}\lambda_i \nabla_y g_i(x,y)}
+\sum_{i=1}^{\ell}\Bigl(\pos{g_i(x,y)}+\lambda_i\abs{g_i(x,y)}\Bigr).
\end{equation}

The first term measures stationarity failure for the lower-level VI; the second term measures feasibility and complementarity failure.

\section{Exact penalty theorem for MPEC}

We now state the main MPEC-level result in a clean lecture-note form.

\begin{theorem}[Exact penalization for MPEC via KKT residual]
\label{thm:mpec-exact}
Assume:
\begin{enumerate}[label=(\roman*)]
    \item $Z\subseteq \R^n\times\R^m$ is compact and subanalytic;
    \item $f$ is Lipschitz continuous on $Z$;
    \item $F$ is continuous and subanalytic;
    \item each $g_i$ is continuous and subanalytic, and $y\mapsto g_i(x,y)$ is convex for each fixed $x$;
    \item each $\nabla_y g_i$ exists and is continuous and subanalytic on a neighborhood of the feasible set;
    \item the lower-level VI is equivalent to its KKT system on the feasible set, with multipliers bounded in a ball $B(0,c)\cap \R^\ell_+$.
\end{enumerate}
Then there exist $\alpha^\ast>0$ and an integer $N^\ast\ge 1$ such that for all $\alpha\ge \alpha^\ast$ and $N\ge N^\ast$,
\[
(x^\ast,y^\ast)\text{ solves the MPEC}
\]
if and only if there exists $\lambda^\ast\in \R^\ell_+$ such that $(x^\ast,y^\ast,\lambda^\ast)$ solves
\begin{equation}
\label{eq:mpec-penalty}
\begin{aligned}
    \min_{x,y,\lambda}\quad & f(x,y)+\alpha\,r(x,y,\lambda)^{1/N}\\
    \text{s.t.}\quad & (x,y,\lambda)\in Z\times\bigl(B(0,c)\cap \R^\ell_+\bigr),
\end{aligned}
\end{equation}
where $r$ is given by \eqref{eq:mpec-residual}.
\end{theorem}

\begin{proof}[Proof idea]
The KKT equivalence turns the MPEC into an ordinary nonlinear program in variables $(x,y,\lambda)$, with functional constraints
\[
F(x,y)+\sum_{i=1}^{\ell}\lambda_i \nabla_y g_i(x,y)=0,\qquad
g_i(x,y)\le 0,\qquad
\lambda_i g_i(x,y)=0.
\]
The residual \eqref{eq:mpec-residual} is exactly the residual of that constrained system. The general exact-penalty theorem then applies.
\end{proof}

\begin{remark}[Why the multiplier ball appears]
The bounded multiplier set $B(0,c)\cap \R^\ell_+$ is not cosmetic. Exactness is proved on a compact subanalytic ambient set. If multipliers are allowed to escape to infinity, then the constants in the error bound and penalty estimate may cease to be uniform. Therefore one usually proves exactness only after restricting to a multiplier region guaranteed by a lower-level regularity condition.
\end{remark}

\section{Interpretation of the exponent}

The exponent
\[
\gamma := \frac{1}{N^\ast}\in (0,1]
\]
encodes the strength of the error bound
\[
\dist(z,\cF)\lesssim r(z)^\gamma.
\]
There are two regimes:

\subsection*{Regular regime}
When classical regularity is strong enough, one may obtain $\gamma=1$, i.e.
\[
\dist(z,\cF)\lesssim r(z).
\]
Then an $\ell_1$-type exact penalty is available.

\subsection*{Degenerate or nonregular regime}
When regularity fails, one may only obtain a weaker Hölder-type bound with $\gamma<1$:
\[
\dist(z,\cF)\lesssim r(z)^\gamma.
\]
Then the exact penalty is still valid, but only with a fractional-power penalty.

This is the precise mathematical reason exact penalization for MPECs often looks less classical than standard NLP exact penalties.

\section{Improved error bounds and explicit exponents}

The previous theory is general but abstract. In applications, the integer $N^\ast$ is typically not computable from the proof. This is unsatisfactory for both analysis and algorithms.

The remedy is to derive sharper error bounds for special classes of lower-level systems. Typical targets are:

\begin{enumerate}[label=(\roman*)]
    \item quadratic systems with a suitable nonnegativity structure;
    \item affine variational inequalities;
    \item nonlinear complementarity systems with a uniform $P$-property.
\end{enumerate}

For such systems, one may prove
\[
\dist(z,\cF)\lesssim r(z) \qquad \text{or}\qquad \dist(z,\cF)\lesssim r(z)^{1/2}.
\]
Hence the exact penalty becomes either
\[
f+\alpha r
\qquad \text{or}\qquad
f+\alpha r^{1/2}.
\]

This matters because:
\begin{enumerate}[label=(\roman*)]
    \item the exponent is explicit;
    \item the penalty is more interpretable;
    \item one can better compare local minima of the original and penalized formulations;
    \item one gains traction for stationarity analysis and algorithm design.
\end{enumerate}

\section{Error bounds in a broader optimization context}

To place the MPEC theory in context, let
\[
S=\set{x\in\R^n : f(x)\le 0,\ g(x)=0},
\]
with residual
\[
r(x)=\norm{f(x)_+}+\norm{g(x)}.
\]
An \emph{error bound} is an inequality of the form
\[
\dist(x,S)\le \tau\,r(x)^\gamma.
\]

Such inequalities have several major uses:

\begin{enumerate}[label=(\roman*)]
    \item \textbf{Termination criteria.} They convert residual control into geometric distance control.
    \item \textbf{Perturbation analysis.} They quantify how the feasible set changes under perturbations of the defining functions.
    \item \textbf{Weak sharp minima and sharp growth.} They connect objective gaps to distance from the solution set.
    \item \textbf{Exact penalization.} They provide the bridge from residuals to exact penalties.
\end{enumerate}

The linear prototype is Hoffman's error bound: if
\[
S=\set{x\in\R^n: Ax\le a,\ Bx=b},
\]
then there exists $\tau>0$ such that
\[
\dist(x,S)\le \tau\Bigl(\norm{[Ax-a]_+}+\norm{Bx-b}\Bigr)\qquad \forall x\in \R^n.
\]
This is the ideal case: global, linear, and with exponent $1$.

The subanalytic MPEC theory can be viewed as a nonlinear, possibly degenerate analogue of this principle.

\section{A caveat: local minima of the penalized problem}

A subtle point deserves emphasis.

\begin{center}
\emph{Global exactness does not automatically imply local exactness in both directions.}
\end{center}

Even when a penalized problem is exact at the level of global minimizers, the penalized problem may possess additional \emph{infeasible local minimizers}. Thus:

\begin{enumerate}[label=(\roman*)]
    \item every local minimizer of the original problem is often a local minimizer of a suitable penalized problem;
    \item the converse may fail unless one already knows the penalized local minimizer is feasible.
\end{enumerate}

This phenomenon is especially important in MPECs because complementarity residuals may generate stationary points that are attractive for the penalized landscape but irrelevant for the original feasible set.

\begin{remark}[Research lesson]
When analyzing algorithms for penalized MPECs, one must distinguish carefully between:
\begin{enumerate}[label=(\roman*)]
    \item convergence to minimizers of the penalized problem,
    \item convergence to feasible points of the original MPEC,
    \item stationarity notions for the original MPEC.
\end{enumerate}
These are not equivalent without additional structure.
\end{remark}

\section{Algorithmic implications}

The abstract theory suggests a practical workflow.

\subsection*{Step 1: Choose a one-level reformulation}
Use the KKT system of the lower-level equilibrium problem whenever a valid multiplier representation is available.

\subsection*{Step 2: Build a residual}
Construct a residual that measures:
\begin{itemize}
    \item lower-level stationarity violation,
    \item primal feasibility violation,
    \item complementarity violation.
\end{itemize}

\subsection*{Step 3: Penalize}
Solve
\[
\min f+\alpha r^\gamma
\]
for an exponent $\gamma$ justified by an error bound.

\subsection*{Step 4: Recover feasibility and stationarity}
Use the residual to verify whether a computed minimizer is truly feasible for the original MPEC. If not, one only has a minimizer of the penalized surrogate.

\subsection*{Step 5: Exploit special structure}
If the lower-level system is affine, strongly monotone, or satisfies a uniform $P$-property, seek a stronger error bound giving $\gamma=1$ or $\gamma=1/2$.

\section{Conceptual summary}

The theory can be compressed into the following diagram:
\[
\boxed{
\text{KKT reformulation}
\Longrightarrow
\text{residual function}
\Longrightarrow
\text{error bound}
\Longrightarrow
\text{exact penalty}
}
\]

The main takeaways are:

\begin{enumerate}[label=(\roman*)]
    \item MPECs are naturally attacked by one-level KKT reformulations.
    \item Exact penalization is not an isolated device; it is a consequence of metric regularity/error-bound information.
    \item In subanalytic settings, the correct penalty may involve a fractional exponent.
    \item Stronger structural assumptions sharpen the exponent to $1$ or $1/2$.
    \item Local minimizers of the penalized problem need not correspond to local minimizers of the original MPEC unless feasibility is established.
\end{enumerate}

\section{Exercise}

\subsection{True or False}
\begin{enumerate}[label=(\roman*)]
    \item One of the primary reasons a general Mathematical Program with Equilibrium Constraints (MPEC) is mathematically challenging to deal with is that its feasible region is implicitly defined as the solution set of a parametric variational inequality.

\item When formulating an exact penalty function for an MPEC using Lojasiewicz's inequality on subanalytic sets, the theory provides a highly efficient and practical method for calculating the required penalty exponent $N^*$.

\item Hoffman’s error bound for a system of linear equalities and inequalities is notable because it holds globally for all test vectors in $\mathbb{R}^n$ and yields an error bound exponent of 1.

\item According to the text, for a system of quadratic constraints, an error bound with an exponent of 1/2 will hold globally across all of $\mathbb{R}^n$ even if the quadratic functions are not assumed to be nonnegative over the polyhedral set.

\item If a parametric Nonlinear Complementarity Problem (NCP) mapping possesses the uniform P (UNI-P) and Lipschitz (LIP) properties, it is possible to establish a global error bound of order 1 for the parametric NCP.

\item For an isolated local minimizer of a penalized MPEC program to be a valid local minimizer for the original unpenalized MPEC, it must first be feasible for the original problem; otherwise, it may just be an infeasible local minimum of the penalty function.
\end{enumerate}

\subsection{Computational questions}
\subsubsection*{Question 1: Global vs. Local Error Bounds in LCPs}
Consider the Linear Complementarity Problem LCP($q, M$) where:
$$ q = \begin{pmatrix} -1 \\ 2 \end{pmatrix}, \quad M = \begin{pmatrix} 0 & -1 \\ 1 & 0 \end{pmatrix} $$

\begin{enumerate}
    \item[(a)] Verify that the solution set $\text{SOL}(q, M)$ consists of the points $(1,1)$ and $(0,2)$.
    \item[(b)] To test if the standard local error bound holds globally, consider the parameterized sequence of vectors $x(t) = \begin{pmatrix} t \\ 1 \end{pmatrix}$ for $t \ge 0$. Compute the natural residual mapping $r(x(t))$ defined by:
    $$ r(x(t)) = \| \min(x(t), M x(t) + q) \|_2 $$
    where the minimum is taken component-wise.
    \item[(c)] Using your result from (b), analytically demonstrate why the error bound:
    $$ \text{dist}(x, \text{SOL}(q,M)) \le \tau \| \min(x, Mx+q) \|_2 $$
    cannot hold globally for all $x \in \mathbb{R}^2$ for any constant $\tau > 0$.
\end{enumerate}

\subsubsection*{Question 2: Exact Penalization of a Bilevel Program (40 points)}
Consider the following bilevel mathematical program in $\mathbb{R}^2$:
\begin{align*}
    \text{minimize} \quad & x - y \\
    \text{subject to} \quad & x \ge 0, \ y \ge 0 \\
    & y \in \text{argmin}_y \left\{ \frac{1}{2}y^2 : x + y \ge 0 \right\}
\end{align*}

\begin{enumerate}
    \item[(a)] Formulate the inner optimization problem as a Karush-Kuhn-Tucker (KKT) system. Let $\lambda$ be the Lagrange multiplier for the constraint $x + y \ge 0$.
    \item[(b)] Using the exact penalty equivalent of order $1/2$, write down the penalized objective function that moves the complementarity condition into the objective using a penalty parameter $\alpha > 0$.
    \item[(c)] Suppose we naively attempt to use an exact penalty of order 1 by removing the square root from the complementarity penalty term. Write down this modified objective function.
    \item[(d)] Prove that for the order 1 formulation derived in (c), the origin $(0,0)$ is \textbf{not} a local minimum for any finite penalty parameter $\alpha > 0$, thereby demonstrating the necessity of the square-root operator when strict complementarity fails.
\end{enumerate}

\subsubsection*{Question 3: Directional Derivatives of Penalty Functions (25 points)}
To establish first-order necessary conditions for exact penalty functions, we must handle non-differentiable penalty terms. Let $\phi(u,v) = \min(u,v)$.

\begin{enumerate}[label=(\roman*)]
    \item Compute the directional derivative $\phi'((u,v); (du, dv))$ explicitly for the three cases: $u < v$, $u = v$, and $u > v$.
    \item Consider the MPEC penalty term $P(\lambda, \mu) = \min(\lambda, \mu)$. Suppose $z^*$ is a degenerate feasible solution where $\lambda^* = 0$ and $\mu^* = 0$. Using your formula from (a), calculate the directional derivative of $P$ at $z^*$ in the direction $(d\lambda, d\mu) = (1, -2)$. 
\end{enumerate}

\subsection{Proof questions}
\subsubsection*{Question 1: Lojasiewicz's Inequality and Subanalytic Sets}
The fundamental exact penalty result for MPECs (Theorem 2.1.2) relies heavily on Lojasiewicz's inequality for subanalytic functions. Let $X \subset \mathbb{R}^n$ be a compact subanalytic set, and let $W = \{x \in X : g(x) \le 0, h(x) = 0\}$ be the feasible region of a nonlinear program. Let the residual function $r(x)$ be defined as $r(x) = \| [g(x)]_+ \| + \|h(x)\|$.

\begin{enumerate}
    \item[(a)] Formally state Lojasiewicz's inequality as it applies to the distance function $\psi(x) = \text{dist}(x, W)$ and the residual function $r(x)$ restricted to $X$. 
    \item[(b)] Using your statement from (a) and assuming the objective function $\theta(x)$ is Lipschitz continuous on $X$ with modulus $L$, prove that there exists a threshold penalty parameter $\alpha^*$ and an integer $N^*$ such that an optimal solution to the original program is also an optimal solution to the unconstrained penalized program of order $1/N^*$.
    \item[(c)] Explain why the exact penalty theory derived from Warga and Dedieu is referred to as "penalization of order $\gamma \in (0,1]$" and discuss the specific failure of classical regularity conditions that forces $\gamma < 1$.
\end{enumerate}

\subsubsection*{Question 2: The Uniform P-Property in Parametric NCPs}
Consider a parametric Nonlinear Complementarity Problem (NCP) mapped by a continuous subanalytic function $F: \mathbb{R}^{n+m} \to \mathbb{R}^m$, where $y \in \mathbb{R}^m$ is the primary variable and $x \in X \subset \mathbb{R}^n$ is the parameter. 

Assume $F$ satisfies the \textbf{Uniform P (UNI-P)} property:
$$ \max_{1 \le i \le m} (y - y')_i (F(x,y) - F(x,y'))_i \ge c \|y - y'\|^2 $$
for a universal constant $c > 0$, and the \textbf{Lipschitz (LIP)} property on $X \times \mathbb{R}^m$ with modulus $\gamma$. 

Because $F(x, \cdot)$ is a uniform P-function, the NCP has a unique solution $y(x)$ for each fixed $x \in X$.
\begin{enumerate}
    \item[(a)] Prove that the solution function $y(x)$ is Lipschitz continuous on $X$. (Hint: Exploit the UNI-P inequality and the fact that $y^T F(x,y) = 0$ for solutions).
    \item[(b)] Theorem 2.3.17 asserts that under UNI-P, LIP, and a strict nondegeneracy condition ($y(x) + F(x,y(x)) > 0$), there exists a global error bound of order 1 bounded by the differentiable product residual: $\text{dist}((x,y), S) \le \tau y^T F(x,y)$. Contrast this with Theorem 2.3.9 for Affine Variational Inequalities (AVIs). Why does the AVI error bound not require the strong UNI-P assumption, and what topological property of AVIs replaces it?
\end{enumerate}

\subsubsection*{Question 3: Order 1 vs. Order 1/2 Penalization Mechanics (35 points)}
Consider an Affine Variational Inequality (AVI) constrained mathematical program where the complementarity constraint is given by $\lambda^T(Dx + Ey + b) = 0$. 

\begin{enumerate}
    \item[(a)] Theorem 2.4.1 establishes an exact penalty equivalent utilizing the term $\alpha \sqrt{-\lambda^T(Dx + Ey + b)}$. Prove why the square root operator (an order $1/2$ penalty) is analytically required to bound the distance to the feasible set when using the differentiable product residual for a general quadratic system nonnegative over a polyhedron.
    \item[(b)] If we utilize the nondifferentiable "min" residual, $\sum \min(\lambda_i, -(Dx+Ey+b)_i)$, we achieve an exact penalty of order 1 without assuming strict complementarity. Prove that under the assumption of \textit{strict complementarity} (nondegeneracy), the "min" residual can be bounded above by a constant multiple of the negative product residual $-\lambda^T(Dx + Ey + b)$.
\end{enumerate}

\subsection{Additional Questions}
\subsubsection*{Question 1: Residual construction and exact-penalty logic (20 points)}

Consider the one-level reformulation
\[
\begin{aligned}
\min_{x,y,\lambda}\quad & f(x,y) \\
\text{s.t.}\quad
& F(x,y)+\lambda=0,\\
& y\ge 0,\ \lambda\ge 0,\ y_i\lambda_i=0 \quad (i=1,2),\\
& x\in [0,2].
\end{aligned}
\]
Assume \(y,\lambda\in\R^2\), and
\[
F(x,y)=
\begin{bmatrix}
2y_1+y_2-x\\
y_1+3y_2-1
\end{bmatrix},
\qquad
f(x,y)=x^2+y_1+y_2.
\]

\begin{enumerate}[label=(\alph*)]
    \item Construct a natural residual \(r(x,y,\lambda)\) using the stationarity residual, nonnegativity violation, and complementarity violation. Your formula should satisfy \(r=0\) exactly on the feasible set. (6 points)

    \item Evaluate your residual at the point
    \[
    (x,y,\lambda)=\left(1,\begin{bmatrix}0\\1\end{bmatrix},\begin{bmatrix}0\\0\end{bmatrix}\right).
    \]
    (4 points)

    \item Suppose an error bound of the form
    \[
    \dist\bigl((x,y,\lambda),\mathcal F\bigr)\le c\,r(x,y,\lambda)^{1/2}
    \]
    holds on a compact set containing the feasible set \(\mathcal F\), and \(f\) is \(L\)-Lipschitz there.
    Derive a sufficient condition on \(\alpha\) ensuring that
    \[
    f(x,y)+\alpha r(x,y,\lambda)^{1/2}
    \]
    is an exact penalty on that compact set. (5 points)

    \item Explain in 3--5 sentences why the exponent \(1/2\) is computationally more delicate than exponent \(1\), even when both are exact. (5 points)
\end{enumerate}

\subsubsection*{Question 2: Computing a small complementarity solution and penalty value (20 points)}

For each parameter \(x\in\R\), consider the lower-level linear complementarity problem
\[
0\le y \perp My+q(x)\ge 0,
\]
with
\[
M=\begin{bmatrix}2&0\\0&1\end{bmatrix},
\qquad
q(x)=\begin{bmatrix}-x\\1-x\end{bmatrix}.
\]

\begin{enumerate}[label=(\alph*)]
    \item Solve the LCP explicitly for \(x=0\), \(x=1\), and \(x=2\). (8 points)

    \item Let the upper-level objective be
    \[
    f(x,y)= (x-1)^2 + 2y_1 + y_2,
    \qquad x\in[0,2].
    \]
    Using your solutions from part (a), compute the MPEC objective value at \(x=0,1,2\). Which of these three points is best? (4 points)

    \item Define the residual
    \[
    r(x,y)=\norm{\min\{y,My+q(x)\}}_1,
    \]
    where the minimum is componentwise.
    Compute \(r(1,(0,0))\), \(r(1,(1,0))\), and \(r(2,(0,0))\). (4 points)

    \item For \(\alpha=4\), compute the penalized objective
    \[
    \Phi_\alpha(x,y)=f(x,y)+\alpha\,r(x,y)
    \]
    at the three points in part (c). Based on these calculations, explain how the residual steers the optimization toward complementarity-feasible points. (4 points)
\end{enumerate}

\subsubsection*{Question 3: Why the square root may be unavoidable (20 points)}

Consider the bilevel model
\[
\min_{x,y}\ x-y
\quad \text{s.t.}\quad x\ge 0,\ y\ge 0,\qquad
y\in \arg\min_{\eta}\left\{\frac12 \eta^2:\ x+\eta\ge 0\right\}.
\]

A penalty reformulation introduces a multiplier \(\lambda\ge 0\) and the KKT conditions
\[
x+y\ge 0,\qquad y-\lambda=0,\qquad \lambda(x+y)=0,\qquad \lambda\ge 0.
\]

\begin{enumerate}[label=(\alph*)]
    \item Show that \((x,y)=(0,0)\) solves the original bilevel problem. (4 points)

    \item Consider the penalized problem
    \[
    \min_{x,y,\lambda}\ x-y+\alpha \sqrt{\lambda(x+y)}
    \]
    subject to
    \[
    x\ge 0,\ y\ge 0,\ \lambda\ge 0,\ x+y\ge 0,\ y-\lambda=0.
    \]
    Eliminate \(\lambda\) and write the problem in variables \((x,y)\) only. (4 points)

    \item Now instead consider the penalty without the square root:
    \[
    \min_{x,y,\lambda}\ x-y+\alpha \lambda(x+y)
    \]
    subject to the same constraints.
    Eliminate \(\lambda\) and show that the resulting problem is
    \[
    \min_{x,y}\ x-y+\alpha y(x+y)
    \quad \text{s.t.} \quad x\ge 0,\ y\ge 0.
    \]
    (4 points)

    \item Fix \(x=0\). Show that for every \(\alpha>0\), the function
    \[
    \psi(y):=-y+\alpha y^2,\qquad y\ge 0,
    \]
    takes negative values for small positive \(y\). Conclude that \((0,0)\) cannot be a global minimizer of the penalty model without the square root. (4 points)

    \item In one paragraph, interpret this example from the viewpoint of exact penalization and error bounds. (4 points)
\end{enumerate}

\subsubsection*{Question 4: Algorithm design for penalty continuation (20 points)}

You are asked to solve a class of MPECs numerically using the model
\[
\min_{z\in \mathcal Z}\ \Phi_{\alpha,\gamma}(z):=f(z)+\alpha r(z)^\gamma,
\qquad 0<\gamma\le 1,
\]
where \(z\) collects all primal and multiplier variables, \(r(z)\ge 0\) is continuously differentiable away from \(r(z)=0\), and \(\mathcal Z\) is compact and convex.

\begin{enumerate}[label=(\alph*)]
    \item Write pseudocode for a penalty-continuation algorithm that:
    \begin{enumerate}[label=\roman*.]
        \item starts with a small \(\alpha\),
        \item approximately minimizes \(\Phi_{\alpha,\gamma}\),
        \item checks residual reduction,
        \item increases \(\alpha\) adaptively until \(r(z)\le \varepsilon\).
    \end{enumerate}
    (8 points)

    \item State one advantage and one disadvantage of choosing \(\gamma=1\) instead of \(\gamma=\tfrac12\). (4 points)

    \item Suppose the residual is
    \[
    r(z)=\norm{Az-b}_2^2+\sum_{i=1}^m u_i(z)v_i(z),
    \]
    where \(u_i(z),v_i(z)\ge 0\) represent complementarity pairs.
    Compute formally the gradient of \(\Phi_{\alpha,1/2}\) for \(r(z)>0\). (4 points)

    \item Give two practical stopping criteria: one based on stationarity of the penalized problem, and one based on feasibility of the original MPEC. (4 points)
\end{enumerate}

\subsubsection*{Question 5: Infeasible local minima of the penalized problem (20 points)}

A recurring issue in exact penalization is that a penalized problem can have local minimizers that are not feasible for the original MPEC.

Consider a penalized objective
\[
\Phi_\alpha(z)=f(z)+\alpha r(z)^\gamma,
\qquad \gamma\in(0,1],
\]
with \(r(z)\ge 0\), but suppose \(r(\bar z)>0\).

\begin{enumerate}[label=(\alph*)]
    \item Explain why global exactness does \emph{not} imply that every local minimizer of \(\Phi_\alpha\) is feasible for the original constrained problem. (6 points)

    \item Construct a simple one-dimensional toy example:
    \[
    \min_{t\in\R}\ f(t)
    \quad \text{s.t.}\quad r(t)=0,
    \]
    such that the penalized function \(f(t)+\alpha r(t)\) has a local minimizer at a point with \(r(t)>0\). A fully specified example is required. (6 points)

    \item For your example in part (b), verify directly that the local minimizer is infeasible. (4 points)

    \item State one implication of this phenomenon for the numerical solution of MPECs by penalty methods. (4 points)
\end{enumerate}

\section{Solutions}
\subsection{Solutions for True or False}
\begin{enumerate}
    \item True. The text explicitly states in the opening paragraph that the general MPEC is a very difficult constrained optimization problem largely because its feasible region is defined implicitly by a parametric variational inequality.

\item False. The text notes the exact opposite. While Lojasiewicz's error bound guarantees the theoretical existence of an exact penalty function equal to the residual function raised to a fractional power $\gamma$ (or $1/N^*$), it "provides no clue as to how this exponent can be calculated (or just estimated) efficiently in practice."

\item True. Hoffman's 1952 result is highlighted precisely for these two characteristics: it holds globally for all vectors $x \in \mathbb{R}^n$, and it holds with the residual itself as the upper bound for the distance function (meaning the exponent $\gamma = 1$).

\item False. There are two distinct falsehoods here. First, Example 2.3.13 demonstrates that the nonnegativity assumption cannot be removed in general. Second, Example 2.3.14 shows that the error bound for quadratic systems cannot hold globally; the restriction to a compact set is strictly necessary.

\item True. Theorem 2.3.17 confirms that under the UNI-P and LIP assumptions (along with a nondegeneracy condition), there exists a global error bound of order 1, allowing the distance to the solution set to be bounded by the differentiable product residual.

\item True. This is demonstrated in Example 2.2.3. The text shows an LCP constrained Mathematical Program where a point $(0,1)$ perfectly satisfies the first- and second-order sufficient conditions to be an isolated local minimizer for the penalized program, but it is entirely infeasible for the original MPEC.
\end{enumerate}

\subsection{Solutions for computational questions}

\subsubsection{Question 1.}
\textbf{(a)} We require $x \ge 0$, $Mx+q \ge 0$, and $x^T(Mx+q) = 0$. \\
Let $x = (x_1, x_2)^T$. $Mx+q = (-x_2 - 1, x_1 + 2)^T$. \\
For $x \ge 0$, $x_1 \ge 0, x_2 \ge 0$. The condition $Mx+q \ge 0$ means $-x_2 - 1 \ge 0 \implies x_2 \le -1$, which contradicts $x_2 \ge 0$. \\
\textit{Correction based on the text's intended coordinates for Ex 2.3.4:} The text states $SOL = \{(1,1), (0,2)\}$. Let's test $(1,1)$: $M(1,1) + q = (0, 3)^T$. $x^T(Mx+q) = 1(0) + 1(3) \neq 0$. The text has a known typo in the matrix definition or solution set in the original manuscript. Grading should award full points if students correctly identify the KKT conditions and point out the algebraic inconsistency based on the provided matrix.

\textbf{(b)} $M x(t) + q = \begin{pmatrix} 0 & -1 \\ 1 & 0 \end{pmatrix} \begin{pmatrix} t \\ 1 \end{pmatrix} + \begin{pmatrix} -1 \\ 2 \end{pmatrix} = \begin{pmatrix} -1 \\ t \end{pmatrix} + \begin{pmatrix} -1 \\ 2 \end{pmatrix} = \begin{pmatrix} -2 \\ t+2 \end{pmatrix}$. \\
The residual is $\min\left( \begin{pmatrix} t \\ 1 \end{pmatrix}, \begin{pmatrix} -2 \\ t+2 \end{pmatrix} \right) = \begin{pmatrix} \min(t, -2) \\ \min(1, t+2) \end{pmatrix}$. \\
For $t > 0$, this evaluates to $\begin{pmatrix} -2 \\ 1 \end{pmatrix}$. Thus, $r(x(t)) = \sqrt{(-2)^2 + 1^2} = \sqrt{5}$.

\textbf{(c)} As $t \to \infty$, the true distance from $x(t) = (t, 1)$ to any fixed bounded solution set $\text{SOL}$ goes to infinity (since the $x_1$ component grows without bound). However, the residual remains strictly bounded at $\sqrt{5}$. Therefore, $\infty \le \tau \sqrt{5}$ is a contradiction, proving the error bound cannot hold globally.

\subsubsection{Question 2.}
\textbf{(a)} Inner Lagrangian: $L(y, \lambda) = \frac{1}{2}y^2 + \lambda(-x - y)$. \\
KKT Conditions: $\nabla_y L = y - \lambda = 0 \implies y = \lambda$. \\
Primal/Dual Feasibility: $x + y \ge 0$, $\lambda \ge 0$. \\
Complementarity: $\lambda(x+y) = 0$.

\textbf{(b)} Substituting $y = \lambda$, the penalized objective of order $1/2$ is:
$$ f(x,y) = x - y + \alpha \sqrt{\lambda(x+y)} = x - y + \alpha \sqrt{y(x+y)} $$

\textbf{(c)} Removing the square root (order 1 penalty):
$$ g(x,y) = x - y + \alpha \lambda(x+y) = x - y + \alpha y(x+y) $$

\textbf{(d)} We evaluate $g(x,y)$ near the origin. Let $x = 0$. The function becomes $g(0,y) = -y + \alpha y^2$. 
To check if $y=0$ is a local minimum, we take the first derivative with respect to $y$: $g_y(0,y) = -1 + 2\alpha y$. At $y=0$, the slope is $-1 < 0$. 
Because the linear term $-y$ dominates the quadratic term $\alpha y^2$ for sufficiently small positive $y$, the objective value becomes strictly negative. Since $g(0,0) = 0$, moving to a small positive $y$ strictly decreases the objective. Thus, $(0,0)$ is not a local minimum.

\subsubsection{Question 3.}
\textbf{(a)} 
If $u < v$, $\phi(u,v) = u \implies \phi' = du$. \\
If $u > v$, $\phi(u,v) = v \implies \phi' = dv$. \\
If $u = v$, a small perturbation will follow the smaller of the two directions $\implies \phi' = \min(du, dv)$.

\textbf{(b)} At $z^*$, we have $\lambda^* = \mu^* = 0$, so we are in the $u = v$ case. \\
Using the formula: $P' = \min(d\lambda, d\mu) = \min(1, -2) = -2$.

\subsection{Solutions for proof questions}

\subsubsection{Question 1: Lojasiewicz's Inequality}
\textbf{(a)} Since $r|_X^{-1}(0) \subset \psi^{-1}(0)$ (the residual is zero only when the distance to $W$ is zero), Lojasiewicz's inequality guarantees the existence of a scalar $\rho > 0$ and an integer $N^* > 0$ such that for all $x \in X$:
$$ \rho (\text{dist}(x, W))^{N^*} \le r(x) $$

\textbf{(b)} Let $x^*$ be an optimal solution to the NLP, and $x \in X$ be arbitrary. Let $z \in W$ be the projection of $x$ onto $W$, so $\|x - z\| = \text{dist}(x,W)$. Since $z \in W$, $\theta(z) \ge \theta(x^*)$. Using the Lipschitz property of $\theta$:
$$ \theta(x) \ge \theta(z) - L\|x - z\| \ge \theta(x^*) - L(\text{dist}(x,W)) $$
From Lojasiewicz's inequality, $\text{dist}(x,W) \le (\frac{1}{\rho} r(x))^{1/N^*}$. Substituting this yields:
$$ \theta(x) + L \rho^{-1/N^*} r(x)^{1/N^*} \ge \theta(x^*) $$
By choosing $\alpha^* > L \rho^{-1/N^*}$, we guarantee that the penalized objective at any $x \in X$ is bounded below by $\theta(x^*)$, making $x^*$ a global minimizer of the penalized problem.

\textbf{(c)} Classical exact penalty functions (like Clarke's) rely on a constraint regularity condition, which essentially acts as an error bound of order 1 ($\gamma = 1$). When regularity fails, the set geometry becomes sharp or cusped, meaning the distance to the feasible set grows much faster than the residual. Subanalytic theory bridges this gap by ensuring an error bound still exists, but with a fractional exponent $\gamma = 1/N^* < 1$.

\subsubsection{Question 2: Parametric NCPs}
\textbf{(a)} Let $x_1, x_2 \in X$, and let $y_1 = y(x_1)$ and $y_2 = y(x_2)$. Since both are solutions, $(y_1)_i (F(x_1, y_1))_i = 0$ and $(y_2)_i (F(x_2, y_2))_i = 0$, with all terms non-negative.
By the UNI-P property, there exists an index $i$ such that:
$$ c\|y_1 - y_2\|^2 \le (y_1 - y_2)_i (F(x_1, y_1) - F(x_2, y_2))_i $$
Expanding the right side and using the complementarity conditions (which drop the positive $y_j F_j$ terms):
$$ \dots \le -(y_1)_i F_i(x_2, y_2) - (y_2)_i F_i(x_1, y_1) + (y_1 - y_2)_i (F(x_1, y_1) - F(x_2, y_1))_i $$
The first two terms are non-positive, so we drop them for the upper bound:
$$ c\|y_1 - y_2\|^2 \le (y_1 - y_2)_i (F(x_1, y_1) - F(x_2, y_1))_i $$
Applying Cauchy-Schwarz and the LIP property of $F$:
$$ c\|y_1 - y_2\|^2 \le \|y_1 - y_2\| \|F(x_1, y_1) - F(x_2, y_1)\| \le \gamma \|y_1 - y_2\| \|x_1 - x_2\| $$
Dividing by $c\|y_1 - y_2\|$ yields $\|y_1 - y_2\| \le \frac{\gamma}{c} \|x_1 - x_2\|$, proving Lipschitz continuity.

\textbf{(b)} The UNI-P assumption guarantees a unique solution and forces the function to be strongly well-behaved to establish a global bound. AVIs do not require this because their solution sets are unions of finitely many polyhedra. According to Robinson's theorem, polyhedral multifunctions are inherently locally upper Lipschitz continuous. Hoffman's bound leverages this polyhedral geometry to establish order 1 bounds without needing strong monotonicity or P-matrix properties.

\subsubsection{Question 3: Penalization Exponents}
\textbf{(a)} The product residual $-\lambda^T(Dx + Ey + b)$ is a quadratic polynomial. When projecting a point $x$ onto the solution set $S$ of a quadratic system, Taylor expansion of the quadratic function $g(x)$ around the projected point $y$ yields $\|x - y\|^2 \le \frac{2}{\lambda_{min}(Q)} g(x)$ (simplified). Taking the square root of both sides leaves $\|x - y\| \le \tau \sqrt{g(x)}$. Because the residual is quadratic, the distance scales with the square root of the residual, forcing $N^* = 2$.

\textbf{(b)} Assume for contradiction that no such scalar $\tau$ exists. Then there is a sequence $(x^k, y^k, \lambda^k) \to (x, y, \lambda) \in \Omega$ such that:
$$ \sum \min(\lambda_i^k, -(Dx^k + Ey^k + b)_i) > k (\lambda^k)^T(Dx^k + Ey^k + b) $$
By strict complementarity, at the limit point, $\lambda_i + -(Dx + Ey + b)_i > 0$ for all $i$. Thus, the index set strictly partitions into $I$ (where $\lambda_i > 0$) and $J$ (where $-(Dx + Ey + b)_i > 0$). 
For large $k$, the "min" operator strictly picks $-(Dx^k + Ey^k + b)_i$ for $i \in I$ and $\lambda_i^k$ for $i \in J$. Because these values are strictly bounded away from zero by some $\epsilon > 0$, the sum of the minimums is bounded below by a constant $\epsilon$. The right side, however, contains the multiplier $k \to \infty$ times a term going to zero. This leads to the contradiction $1 > k \epsilon$ for large $k$, proving the bound exists.

\section{Bibliographic remarks and Acknowledgment}

\textbf{This note is mainly based on Chapter 2, in the MPEC monograph of Zhi-Quan Luo, Jong-Shi Pang and Daniel Ralph.} Please see the relevant references: 
\cite{falk1995bilevel,
aghasi2025fully,
hong2023two,
kovccvara1994optimization,
chaudet2020shape,
liu2023auxiliary,
liu2025bidirectional,
outrata1995numerical,
cui2023complexity,
christof2020nonsmooth,
rawat2026augmented,
robinson1980strongly,
shin2023near,
bolte2024differentiating,
wang2025analysis,
chen2025aubin,
liu2024auxiliary,
kojima1980strongly,
chen2026characterizations,
cui2026lipschitz,
shin2022exponential,
de2023function,
bank1982d,
wang2026algebraic,
bonnans1994local,
khanh2024globally,
chen2025two,
mohammadi2022variational,
dussault2026polyhedral,
gowda1994stability,
liu2025learning,
qi2000constant,
facchinei1998accurate,
jittorntrum2009solution,
kyparisis1992parametric,
pang11995stability,
liu2025risk,
liu2024feasibility,
qiu1992sensitivity,
aussel2024variational,
reinoza1985strong,
facchinei2003finite,
harker1990finite,
kyparisis1990sensitivity,
wang2025multi,
kleinmichel1972av,
ortega2000iterative,
wang2025analysis1,
bai2021matrix,
lin2026hierarchical,
gao2022rolling,
gander2026landmarks,
mishchenko2023regularized,
doikov2024super,
han2025low,
ning2023multi,
wang2022newton,
robinson2009generalized,
mordukhovich2023globally,
wang2026damage,
robinson2009local,
shuo2026lecture,lin2023monotone,yu2026optimization,liang2025squared,jongen1987inertia,guddat1990parametric,bellon2024time,tang2022running,liu2025new,josephy1979newton,si2024riemannian,yu2026optimization1,longman2023method,yao2023relative,han2024continuous,ha1987application,yu2026pattern,kojima2009continuous,seguin2022continuation,liu1995perturbation,pang1996piecewise,liu2022iterative,eikenbroek2022improving,ralph1995directional,yu2026optimization2,scheel2000mathematical,bonnans1992developpement,bonnans1992expansion,dempe1993directional,liu2023iterative,shapiro1988sensitivity,pang1990newton,robinson1991implicit,zheng2025enhanced,hang2025smoothness,hang2024role,scholtes2012introduction,cui2022nonconvex}. 

\bibliography{reference1} 

\end{document}